\documentclass[12pt]{amsart}

\usepackage{amsmath,amssymb,amsbsy,amsfonts,amsthm,latexsym,
                        amsopn,amstext,amsxtra,euscript,amscd,mathrsfs,color,bm,cite}
                       
\usepackage{float} 
\usepackage[english]{babel}
\usepackage{mathtools}
\usepackage{todonotes}
\usepackage{url}
\usepackage[colorlinks,linkcolor=blue,anchorcolor=blue,citecolor=blue,backref=page]{hyperref}
\usepackage{todonotes}

\RequirePackage{mathrsfs} \let\mathcal\mathscr

\usepackage{enumitem}

\def\le{\leqslant}
\def\ge{\geqslant}

\usepackage{mathtools}
\usepackage{todonotes}
\usepackage[norefs,nocites]{refcheck}

\usepackage[english]{babel}
\usepackage{mathtools}
\usepackage{todonotes}
\usepackage{url}
\usepackage[colorlinks,linkcolor=blue,anchorcolor=blue,citecolor=blue,backref=page]{hyperref}

\begin{document}

\newtheorem{theorem}{Theorem}
\newtheorem{lemma}[theorem]{Lemma}
\newtheorem{claim}[theorem]{Claim}
\newtheorem{cor}[theorem]{Corollary}
\newtheorem{prop}[theorem]{Proposition}
\newtheorem{remark}[theorem]{Remark}
\newtheorem{example}[theorem]{Example}
\newtheorem{definition}{Definition}
\newtheorem{question}[theorem]{Question}
\newtheorem{conj}[theorem]{Conjecture}
\newcommand{\hh}{{{\mathrm h}}}

\numberwithin{equation}{section}
\numberwithin{theorem}{section}
\numberwithin{table}{section}

\numberwithin{figure}{section}

\def\sssum{\mathop{\sum\!\sum\!\sum}}
\def\ssum{\mathop{\sum\ldots \sum}}
\def\dsum{\mathop{\quad \sum \qquad \sum}}
\def\iint{\mathop{\int\ldots \int}}

%\def\squareforqed{\hbox{\rlap{$\sqcap$}$\sqcup$}}
%\def\qed{\ifmmode\squareforqed\else{\unskip\nobreak\hfil
%\penalty50\hskip1emrll\nobreak\hfil\squareforqed
%\parfillskip=0pt\finalhyphendemerits=0\endgraf}\fi}%%

%  use the AMS-Euler Fraktur fonts
%%%%%%%%%%%%%%%%%%%%%%%%%%%%%%%%%%
\newfont{\teneufm}{eufm10}
\newfont{\seveneufm}{eufm7}
\newfont{\fiveeufm}{eufm5}
%%%%%%%%%%%%%%%%%%%%%%%%%%%%%%%%%
%
%  allow automatic size selection in math mode
%
%%%%%%%%%%%%%%%%%%%%%%%%%%%%%%%%%
\newfam\eufmfam
     \textfont\eufmfam=\teneufm
\scriptfont\eufmfam=\seveneufm
     \scriptscriptfont\eufmfam=\fiveeufm
%%%%%%%%%%%%%%%%%%%%%%%%%%%%%%%%%
%
%  \frak works on a single symbol at a time...
%
\def\frak#1{{\fam\eufmfam\relax#1}}

\newcommand{\bflambda}{{\boldsymbol{\lambda}}}
\newcommand{\bfmu}{{\boldsymbol{\mu}}}
\newcommand{\bfxi}{{\boldsymbol{\xi}}}
\newcommand{\bfrho}{{\boldsymbol{\rho}}}

\def\fA{{\mathfrak A}}
\def\fB{{\mathfrak B}}
\def\fC{{\mathfrak C}}
\def\fK{{\mathfrak K}}
\def\fM{{\mathfrak M}}
\def\fS{{\mathfrak S}}
 \def\fW{{\mathfrak W}}
  \def\fU{{\mathfrak U}}

\def \balpha{\bm{\alpha}}
\def \bbeta{\bm{\beta}}
\def \bgamma{\bm{\gamma}}
\def \blambda{\bm{\lambda}}
\def \bchi{\bm{\chi}}
\def \bphi{\bm{\varphi}}
\def \bpsi{\bm{\psi}}
\def \bomega{\bm{\omega}}
\def \btheta{\bm{\vartheta}}

\def \bzeta{\bm{\zeta}}
\def \bxi{\bm{\xi}}

\def\eqref#1{(\ref{#1})}

\def\vec#1{\mathbf{#1}}

\def\vk{\vec{k}}
\def\vX{\vec{X}}

\def \LLN  {\prec \hskip-6pt \prec}
\def \GGN {\succ \hskip-6pt \succ}

%%%%%%%%%%%%%%%%%%%%%%%%%
% Alphabet calligraphie %
%%%%%%%%%%%%%%%%%%%%%%%%%
\def\cA{{\mathcal A}}
\def\cB{{\mathcal B}}
\def\cC{{\mathcal C}}
\def\cD{{\mathcal D}}
\def\cE{{\mathcal E}}
\def\cF{{\mathcal F}}
\def\cG{{\mathcal G}}
\def\cH{{\mathcal H}}
\def\cI{{\mathcal I}}
\def\cJ{{\mathcal J}}
\def\cK{{\mathcal K}}
\def\cL{{\mathcal L}}
\def\cM{{\mathcal M}}
\def\cN{{\mathcal N}}
\def\cO{{\mathcal O}}
\def\cP{{\mathcal P}}
\def\cQ{{\mathcal Q}}
\def\cR{{\mathcal R}}
\def\cS{{\mathcal S}}
\def\cT{{\mathcal T}}
\def\cU{{\mathcal U}}
\def\cV{{\mathcal V}}
\def\cW{{\mathcal W}}
\def\cX{{\mathcal X}}
\def\cY{{\mathcal Y}}
\def\cZ{{\mathcal Z}}
\newcommand{\rmod}[1]{\: \mbox{mod} \: #1}

\def\cg{{\mathcal g}}

\def\vr{\mathbf r}

\def\e{{\mathbf{\,e}}}
\def\ep{{\mathbf{\,e}}_p}
\def\eq{{\mathbf{\,e}}_q}

\def\Tr{{\mathrm{Tr}}}
\def\Nm{{\mathrm{Nm}}}

 \def\SS{{\mathbf{S}}}

\def\lcm{{\mathrm{lcm}}}

\def\({\left(}
\def\){\right)}
\def\fl#1{\left\lfloor#1\right\rfloor}
\def\rf#1{\left\lceil#1\right\rceil}

\def\mand{\qquad \mbox{and} \qquad}

%\newcommand{\commI}[2][]{\todo[#1,color=green!60]{I: #2}}
%\newcommand{\commB}[2][]{\todo[#1,color=blue!60]{B: #2}}

%%%%%%%%%%%%%%%%%%%%%%%%%%%%%%%%%%%%%%%%%%%%%%%%%%%%%%%%
%%%%%%%%%%%%%%%%%%%%%%%%%%%%%%%%%%%%%%%%%%%%%%%%%%%%%%%%
%%%%%%%%%%%%%%%%%%%%%%%%%%%%%%%%%%%%%%%%%%%%%%%%%%%%%%%%
%%%%%%%%%%%%%%%%%%%%%%%%%%%%%%%%%%%%%%%%%%%%%%%%%%%%%%%%

%%%%%%%  END OF STANDARD STUFF %%%%%%%%%

%%%%%%%%%%%%%%%%%%%%%%%%%%%%%%%%%%%%%%%%%%%%%%%%%%%%%%%%
%%%%%%%%%%%%%%%%%%%%%%%%%%%%%%%%%%%%%%%%%%%%%%%%%%%%%%%%
%%%%%%%%%%%%%%%%%%%%%%%%%%%%%%%%%%%%%%%%%%%%%%%%%%%%%%%%
%%%%%%%%%%%%%%%%%%%%%%%%%%%%%%%%%%%%%%%%%%%%%%%%%%%%%%%
%%%%%%%%%%%
%%% Spell

\hyphenation{re-pub-lished}

\mathsurround=1pt

\def\bfdefault{b}

\def \A{{\mathbb A}}
\def \P{{\mathbb P}}

\def \F{{\mathbb F}}
\def \K{{\mathbb K}}
\def \N{{\mathbb N}}
\def \Z{{\mathbb Z}}
\def \Q{{\mathbb Q}}
\def \R{{\mathbb R}}
\def \C{{\mathbb C}}
\def\Fp{\F_p}
\def \fp{\Fp^*}

\def\Kmnp{\cK_p(m,n)}
\def\Kmnq{\cK_q(m,n)}
\def\KKap{\cH_p(a)}
\def\KKaq{\cH_q(a)}
\def\KKmnp{\cH_p(m,n)}
\def\KKmnq{\cH_q(m,n)}

\def\Kl{{\mathsf K}}

\def\Klmnp{\cK_p(\ell, m,n)}
\def\Klmnq{\cK_q(\ell, m,n)}

\def \SALMNq {\cS_q(\balpha;\cL,\cI,\cJ)}
\def \SALMNp {\cS_p(\balpha;\cL,\cI,\cJ)}

\def \SACXIQX {\fS_h(\balpha,\bzeta, \bxi; \cI, Q,X)}

\def \balpha{\bm{\alpha}}
\def \bbeta{\bm{\beta}}
\def \bgamma{\bm{\gamma}}
\def \blambda{\bm{\lambda}}
\def \bchi{\bm{\chi}}
\def \bphi{\bm{\varphi}}
\def \bpsi{\bm{\psi}}

\newcommand{\commA}[2][]{\todo[#1,color=green!60]{A: #2}}
\newcommand{\commD}[2][]{\todo[#1,color=magenta!60]{D: #2}}
\newcommand{\commI}[2][]{\todo[#1,color=red!60]{I: #2}}
\newcommand{\commS}[2][]{\todo[#1,color=blue!60]{S: #2}}

\def\SIJq{S_q(\balpha; \cI,\cJ)}
\def\SIJp{S_p(\balpha; \cI,\cJ)}

\def\UMNp{U_p(\balpha, \bbeta; M,N)}
\def\VMNp{V_p(\balpha, \bbeta; M,N)}
\def\LG#1{\(\frac{#1}{p}\)}

\def\SMJq{S_q(\balpha; \cM,\cJ)}
\def\SMJp{S_p(\balpha; \cM,\cJ)}

\def\WIKq{W_{a,q}(\bgamma; \cI,K)}
\def\WIKp{W_{a,p}(\bgamma; \cI,K)}
\def\WMKq{W_{a,q}(\bgamma; \cM,K)}
\def\WMKp{W_{a,p}(\bgamma; \cM,K)}

\def\WaIKq{W_{a,q}(\balpha, \bgamma; \cI;K)}
\def\WaIKp{W_{a,p}(\balpha, \bgamma; \cI,K)}
\def\WaMKq{W_{a,q}(\balpha, \bgamma; \cM;K)}
\def\WaMKp{W_{a,p}(\balpha, \bgamma; \cM;K)}
\def\WaMKKq{\widetilde{W}_{a,q}(\balpha, \bgamma; \cM,\cK)}
\def\WaMKKp{\widetilde{W}_{a,p}(\balpha, \bgamma; \cM,\cK)}

\def\WaIKKq{\widetilde{W}_{a,q}(\balpha, \bgamma; \cI,\cK)}
\def\WaIKKp{\widetilde{W}_{a,p}(\balpha, \bgamma; \cI,\cK)}

\def\WAMNQ{W_q(\balpha, \bgamma; M,N,Q)}

\def\RIJp{\cR_p(\cI,\cJ)}
\def\RIJq{\cR_q(\cI,\cJ)}

\def\TWXJp{\cT_p(\bomega;\cX,\cJ)}
\def\TWXJq{\cT_q(\bomega;\cX,\cJ)}
\def\TWpXJp{\cT_p(\bomega_p;\cX,\cJ)}
\def\TWqXJq{\cT_q(\bomega_q;\cX,\cJ)}
\def\TWJq{\cT_q(\bomega;\cJ)}
\def\TWqJq{\cT_q(\bomega_q;\cJ)}

\newcommand{\supp}{\operatorname{supp}}

  \def \kbar{k^{-1} }
 \def \xbar{x^{-1} }
  \def \ybar{y^{-1} }

\title[Sparse representations of squares]{On sparsity of  representations of polynomials
as linear combinations of exponential functions} 

\author[D. Ghioca]{Dragos Ghioca}
\author[A. Ostafe]{Alina Ostafe}
\author[S. Saleh]{Sina Saleh}
\author[I. E. Shparlinski]{Igor E. Shparlinski}

\keywords{squares, sparse representations}
\subjclass[2010]{Primary 11B37, Secondary 11G25, 37P55}

\address{
Dragos Ghioca\\
Department of Mathematics\\
University of British Columbia\\
Vancouver, BC V6T 1Z2\\
Canada
}

\email{dghioca@math.ubc.ca}

\address{Alina Ostafe\\
School of Mathematics and Statistics\\
University of New South Wales\\
Sydney NSW 2052\\
Australia
}

\email{alina.ostafe@unsw.edu.au}

\address{
Sina Saleh\\
Department of Mathematics\\
University of British Co\-lumbia\\
Vancouver, BC V6T 1Z2\\
Canada
}

\email{sinas@math.ubc.ca}

\address{Igor E. Shparlinski\\ 
School of Mathematics and Statistics\\
University of New South Wales\\
Sydney NSW 2052\\
Australia
}

\email{igor.shparlinski@unsw.edu.au}

\keywords{Squares, }
\subjclass[2010]{11A63, 11B57}

\begin{abstract} 
Given an integer $g$ and also some given integers $m$ (sufficiently large) and  $c_1,\dots, c_m$, we show that the number of all non-negative integers $n\le M$ with the property that there exist non-negative integers $k_1,\dots, k_m$ such that $$n^2=\sum_{i=1}^m c_i g^{k_i}$$ is $o\left(\(\log M\)^{m-1/2}\right)$. We also obtain a similar bound when dealing with more general inequalities $$\left|Q(n)-\sum_{i=1}^m c_i\lambda^{k_i}\right|\le B,$$ where $Q\in \C[X]$ and also $\lambda\in\C$ (while $B$ is a real number).
\end{abstract} 

\maketitle

\section{Introduction}
\label{sec:intro}

\subsection{Set-up} 
Motivated by applications to the dynamical Mordell-Lang conjecture (for more details on this open problem in arithmetic dynamics, we refer the reader to~\cite{BGT}), the authors~\cite{GOSS} have recently 
considered the question about representations of values  of 
polynomials $Q \in \overline{\mathbb{Q}}[X]$ as fixed linear combinations of powers of a prime $p$.
In particular, it is shown in~\cite{GOSS}  that for fixed coefficients $c_1, \ldots, c_m \in \overline{\mathbb{Q}}$
and integral exponents $a_1, \ldots, a_m$ 
the number of positive integers  $n \le N$ for  which $Q(n)$  can be represented as 
$$
Q(n) =  \sum_{i=1}^m c_i p^{a_ik_i}
$$
with some $k_1, \ldots, k_m\in \Z$ is bounded by $O\(\(1+\log N\)^m\)$
where the implied constant depends only on the initial data. 
In fact it is easy to see that for $Q(n) = n$, this bound is tight.  Furthermore, 
a similar result is given in~\cite{GOSS} for representations of the form
\begin{equation}
\label{eq:Qciglambdaj}
Q(n) = \sum_{i=1}^m \sum_{j=1}^s c_{i,j}\lambda_j^{a_ik_i},
\end{equation}
with algebraic integers  
$\lambda_1,\ldots, \lambda_s$, each one of them of absolute value equal to $q$ or 
$\sqrt{q}$ (where $q$ is a given power of a prime number).

Here we first consider representations of the form~\eqref{eq:Qciglambdaj}
with $s=1$ but arbitrary complex (rather than algebraic) parameters. We also generalise 
this to approximations of polynomials rather than precise equalities, that is, 
we consider inequalities of the form 
\begin{equation}
\label{eq:QciglambdaB}
\left|Q(n)- \sum_{i=1}^m c_i \lambda^{k_i}\right|\le B, 
\end{equation}
with $Q \in \C[X]$, $c_1, \ldots, c_m ,\lambda \in \C$ and some $B\in \R$. 

\subsection{Notation} 
We now recall that the notations $A= O(B)$, $A \ll B$ and $B \ge A$ 
are all equivalent to the inequality $|A| \le cB$ with some constant $c$.
Throughout this work all implied constants may depend on the polynomial $Q$ and the 
parameters  $c_i$, $i=1, \ldots, m$ and  $\lambda$ in~\eqref{eq:QciglambdaB}  and
also on  $g$ in~\eqref{eq:Qcig} below. 

For a finite set $\cS$ we use $\# \cS$ to denote its cardinality.

\subsection{New results} \label{sec:res}
We remark that the argument of~\cite{GOSS} is based on a result of 
Laurent~\cite[Th\'{e}or\`{e}me~6]{Lau},  which required all parameters to be 
defined over a number field; furthermore, the result of~\cite{Lau} refers to equalities, not inequalities. Hence here we use a different approach to establish the 
following result. 

\begin{theorem}
\label{thm:Qn lambda B}
Let $c_1, \ldots, c_m, \lambda \in \C$,  $B \in \R$ and let $Q\in \C[X]$ be a non-constant polynomial. 
Then for $N \ge 2$ we have 
$$
\# \{n \le N:~\text{\eqref{eq:QciglambdaB} holds for some}\ k_1, \ldots, k_m\in \Z\}
\ll (\log N)^m.
$$
\end{theorem}

We  observe that the implied constant in Theorem~\ref{thm:Qn lambda B} is effectively computable
in terms of the sizes of the initial data, 
while in the result of~\cite{GOSS} it is not. 

We also note (see Example~\ref{ex:transcendental}) that if one  considers inequalities of the form
\begin{equation}
\label{eq:general inequality}
\left|Q(n) - \sum_{i=1}^m \sum_{j=1}^s c_{i,j}\lambda_j^{a_ik_i}\right|\le B,
\end{equation}
for some arbitrary complex numbers $\lambda_j$, then one cannot expect a similar result as in Theorem~\ref{thm:Qn lambda B}. More precisely, there exists $\lambda\in\C$ such that for any $\varepsilon > 0$ and \emph{each} sufficiently large integer $n$, there exists some positive integer $k_n$ with the property that
$$
\left|n - \frac{i}{\pi}\cdot \left(2^{k_n}-\lambda^{k_n}\right)\right| \le \varepsilon,
$$ 
see Example~\ref{ex:transcendental} for more details. 

Furthermore, we consider the case of  perfect squares and study relations of the form
\begin{equation}
\label{eq:Qcig}
n^2=  \sum_{i=1}^m c_i g^{k_i}, 
\end{equation}
with non-zero integer coefficients $c_1, \ldots, c_m$ and  an integer  basis
$g \ge 2$. Using the square-sieve of Heath-Brown~\cite{HB} 
we improve the exponent $m$ of $\log N$.

For $m=2$,  we write the equation~\eqref{eq:Qcig} as $n^2 = g^{k_1}\(c_1 +  c_2 g^{k_2-k_1}\)$
(with $k_2 \ge k_1$). 
Hence either $c_1 +  c_2 g^{k_2-k_1}$ or $c_1g +  c_2 g^{k_2-k_1+1}$ is a
perfect square $j^2$ for some $j \le n \le N$. Since the largest prime divisor of $j^2 + c$ 
for any $c \ne 0$ tends to infinity with $j$, see~\cite{Kea},  we see that 
$k_2-k_1$ can take only finitely many values. Hence for $m=2$ we have  $O(\log N)$ solution 
to~\eqref{eq:Qcig} with $n\le N$. 
This  bound  is obviously the 
best possible as the example  of the numbers  $2^{k_1} +  2^{k_2}$, with $k_2=k_1+3$ and even 
$k_1$, shows.   We also note that in~\cite[Theorem~5.1~(B)]{CGSZ}, it is established even more generally the \emph{precise} set of all positive integers $n$ for which $u_n$ is of the form $c_1g^{k_1}+c_2g^{k_2}$ (for some given $g$, $c_1$, $c_2$), where $\{u_n\}_{n\ge 1}$ is an arbitrary linear recurrence sequence (the result of~\cite[Theorem~5.1~(B)]{CGSZ} is stated only when $g=p$ is a prime number, but as remarked in~\cite[Section~5]{CGSZ}, the method extends verbatim to an arbitrary integer $g$).

So we are mostly interested in  the case of $m \ge 3$; furthermore, we note that for $m=3$ (and in some cases, depending on $g$ and the $c_i$, even for $m=4$), more precise results are available in the literature (see~\cite{CZ13}). However, when $m\ge 5$, it is very difficult to find a precise description of all $n\in\mathbb{N}$ such that $n^2$ is of the form~\eqref{eq:Qcig} (for some given integers $g$ and $c_i$).

\begin{theorem}
\label{thm:Qn g}
Let $m \ge 3$ and let $c_1, \ldots, c_m$ and 
$g \ge 2$ be   integers.  
Then for $N \ge 2$ we have 
$$
\# \{n \le N:~\text{\eqref{eq:Qcig} holds for some}\ k_1, \ldots, k_m\in \Z\}
\le (\log N)^{m-\gamma_m + o(1)},
$$
where 
$$
\gamma_3 =  \frac{677}{1969} \mand 
\gamma_m = \frac{677m}{1323m +1354} \quad \text{for} \ m\ge 4 . 
$$ 
\end{theorem}

We observe that $\gamma_m \to 677/1323 = 0.5117 \ldots $ as $m \to \infty$. Thus for large 
$m$  Theorem~\ref{thm:Qn g}  saves more than $1/2$ compared to the general bound of  Theorem~\ref{thm:Qn lambda B}.
More precisely, simple calculations show that  $\gamma_m > 1/2$ for $m \ge 44$. 

We remark that the proof of Theorem~\ref{thm:Qn g} is based on some  ideas and results
from~\cite{LuSh}, later enhanced in~\cite{BaSh}. The numerical constants come from the work
of Baker and Harman~\cite{BaHa} on large prime divisors of shifted primes. 

Furthermore, as in~\cite{BaSh} we observe that under the Generalised Riemann Hypothesis
we can obtain a slightly larger value of $\gamma_m$.  

On the other hand defining $s$ as the largest integer with  $s(s+1)/2 \le m$ and considering numbers
$\(g^{h_1} + \ldots + g^{h_s}\)^2$  with 
$$
h_i \le \frac{\log \left(N/s^2\right)}{2\log g}, \quad i =1, \ldots, s,
$$
we see that for at least one choice $c_1, \ldots, c_t> 0$ with $t \le m$ 
and $ c_1+ \ldots + c_t \le s(s+1)/2$ occurs at least $C_0 (\log N)^s$ times,
for some constant $C_0>0$, which shows that the best possible exponent in 
any result of the type of Theorem~\ref{thm:Qn g} must grow with $m$ (at least as 
about $\sqrt{2m}$ for large $m$).

Note that cycling over all $g^m$ choices of 
$$(c_1, \ldots, c_m) \in \{0, \ldots, g-1\}^m
$$
we obtain from Theorem~\ref{thm:Qn g}  a result about the sparsity of the 
values of $n$ for which $n^2$ has at most $m$ non-zero digits to base $g$. 
Various finiteness results on sparse digital representations of perfect powers
can be found in~\cite{BeBu, BBM, CZ13, Mos}. 
Note that as we have just seen, in our setting of arbitrary $m$ 
no finiteness result is  possible, and hence we can use Theorem~\ref{thm:Qn g}
to  provide  a counting result related to
such representations. More precisely we have the following straightforward consequence:
\begin{cor}
Let $m\ge 3$ and let $K\ge 1$ and $g\ge 2$ be integers. Then there are at most $K^{m-\gamma_m + o(1)}$ 
integer squares with $g$-ary expansion of length $K$ and with at most $m$ non-zero digits. 
\end{cor}

\section{Proof of  Theorem~\ref{thm:Qn lambda B}}

\subsection{Counterexample to a possible extension to~\eqref{eq:general inequality}} 
Before proceeding to the proof of Theorem~\ref{thm:Qn lambda B}, we provide the Example~\ref{ex:transcendental} (mentioned in~Section~\ref{sec:res}), which shows that one cannot expect to generalise Theorem~\ref{thm:Qn lambda B} 
to~\eqref{eq:general inequality}, that is, to the case when we approximate $Q(n)$ with a sum of powers of different $\lambda_j$.

\begin{example}
\label{ex:transcendental}
We consider the sequence of positive integers $\{b_j\}_{j\ge 2}$ given by 
$$b_2=2 \mand b_{j+1}=2^{b_j}+b_j+1\text{ for }j\ge 2.$$ 
We let $\lambda=2\cdot e^{2\pi \alpha i}$, where  
$$
\alpha =\sum_{j=2}^\infty \frac{j-1}{2^{b_j}}.
$$
We let $n$ be a positive integer and show that 
\begin{equation}
\label{eq:N large small}
 n-\frac{i}{\pi}\cdot \left(2^{2^{b_n}}-\lambda^{2^{b_n}}\right)  \ll \frac{n}{ b_{n+1}} . 
\end{equation}
Indeed, we first notice that
$$\lambda^{2^{b_n}}=2^{2^{b_n}}\cdot e^{2\pi t_n i},$$
where 
$$
t_n=\sum_{j=n+1}^\infty \frac{j-1}{2^{b_{j}-b_n}}.
$$ 
Then
\begin{align*}
2^{2^{b_n}}-\lambda^{2^{b_n}}& =2^{2^{b_n}}\cdot \left(\left(1-\cos(2\pi t_n)\right) - i\sin\left(2\pi t_n\right)\right)\\
& = 2^{2^{b_n}}\cdot 2\sin(\pi t_n)\cdot \left(\sin(\pi t_n)-i\cos(\pi t_n)\right), 
\end{align*}
and so,
\begin{equation}
\label{eq:2 - lambda}
\frac{i}{\pi}\cdot \left(2^{2^{b_n}}-\lambda^{2^{b_n}}\right) = \frac{2^{1+2^{b_n}}\sin(\pi t_n)}{\pi}\cdot e^{\pi t_ni}.
\end{equation}
Now, by the definition of the rapidly increasing sequence $\{b_j\}_{j\ge 2}$, we have that 
\begin{equation}
\label{eq:t_N}
0<t_n-\frac{n}{2^{b_{n+1}-b_n}}< \frac{1}{2^{2^{b_{n+1}}}};
\end{equation}
also, clearly, $t_n\to 0$ as $n\to\infty$. Furthermore, we know that when $t$ is close to $0$, then 
\begin{equation}
\label{eq:sin t}
|\sin t-t|\le t^2.
\end{equation}
So, using the inequalities~\eqref{eq:t_N} and~\eqref{eq:sin t}, along with the fact that $b_{n+1}=2^{b_n}+b_n+1$, we get that
\begin{equation}
\label{eq:moduli N}
\left|n-\frac{2^{1+2^{b_n}}\sin\left(\pi t_n\right)}{\pi}\right|< 2^{-b_n}  \ll 1/b_{n+1} 
\end{equation}
for all $n$ sufficiently large. Also, for $n$ large, using the inequality~\eqref{eq:t_N} we have that 
\begin{equation}
\label{eq:exp t_N}
\left|1-e^{\pi t_ni}\right|< 2^{-b_n} \ll 1/b_{n+1} .
\end{equation}
Recalling~\eqref{eq:2 - lambda} and using~\eqref{eq:moduli N} and~\eqref{eq:exp t_N},  we derive~\eqref{eq:N large small}. 
 \end{example}

Therefore, the conclusion of Theorem~\ref{thm:Qn lambda B} cannot be generalised to inequalities~\eqref{eq:general inequality} where we approximate a polynomial $Q(n)$ with sums of powers of different $\lambda_1, \ldots, \lambda_s$.

Next we proceed to proving Theorem~\ref{thm:Qn lambda B}.

\subsection{Preliminaries}
We first note that if $|\lambda|= 1$, then the inequality~\eqref{eq:QciglambdaB} yields that $|Q(n)|$ is uniformly bounded above and therefore, we can only have finitely many $n\in\N$ satisfying such inequality since $Q$ is a non-constant polynomial. Furthermore, since the exponents $k_i$ appearing in the inequality~\eqref{eq:QciglambdaB} are arbitrary integers, then without loss of generality, we may assume from now on that $|\lambda|>1$.

Now, if some of the exponents $k_i$, $i=1, \ldots, m$, from~\eqref{eq:QciglambdaB} were non-positive, then the absolute value of the corresponding terms $c_i \lambda^{k_i}$ is uniformly bounded above. So, at the expense of replacing $B$ by a larger constant (but depending only on the absolute values of the $c_i$), we may assume from now on, that each exponent $k_i$ from~\eqref{eq:QciglambdaB} is positive.

Let $Q(X) = a_dX^d + \cdots + a_1X + a_0$ for complex numbers $a_0, \ldots, a_d$ with $a_d\ne 0$. 
There exists $N_0\ge 0$ (depending only on $d$ and the absolute values of the coefficients of $Q$) such that 
\begin{equation}
\label{eq:Q growth}
|Q(n)| \le 2\left|a_d n^d\right|\le N_0\cdot n^d\text{ for each }n\ge N_0.
\end{equation}
Furthermore, at the expense of replacing $N_0$ by a larger positive integer (but still depending only on $d$, the absolute values of the coefficients of $Q$ and also depending on $B$ in this case), we may also assume that 
\begin{equation}
\label{eq:spaced-out}
|Q(n_1+n_2)-Q(n_1)|>2B\text{ for each }n_1,n_2\ge N_0.
\end{equation}

\subsection{Induction}
We proceed to prove our desired result by induction on $m$.

We prove first the base case $m = 1$, which also constitutes the inspiration for our proof for the general case in Theorem~\ref{thm:Qn lambda B}. So, we have that $|Q(n)|= O(N^d)$ for each $0\le n\le N$ (see also the inequality~\eqref{eq:Q growth}). Therefore, for $n\le N$ satisfying the inequality 
\begin{equation}
\label{eq:inequality case 1}
\left|Q(n)-c_1\lambda^{k_1}\right|\le B
\end{equation}
one has $|\lambda|^{k_1} = O\(N^d\)$ and thus, $k_1 = O\(\log N\)$. On the other hand, for a given $k_1$, the inequality~\eqref{eq:inequality case 1} is satisfied by $O(1)$ non-negative integers $n$ (see also~\eqref{eq:spaced-out}), thus proving the desired bound  in the case $m=1$.

So, suppose that the result is true for $m \le s$ and we  prove that Theorem~\ref{thm:Qn lambda B} holds when $m=s+1$; clearly we may assume each $c_i$ for $i=1,\ldots,s+1$  is nonzero. Since there are $m$ powers of $\lambda$ in the inequality~\eqref{eq:QciglambdaB}, then in order to prove Theorem~\ref{thm:Qn lambda B}, it suffices to prove that the set $\cS_0$, consisting of all $n\in \N$ for which there exist integers 
\begin{equation}
\label{eq:inequalities}
1\le k_1\le k_2\le \cdots \le k_{s+1}
\end{equation} 
 such that 
\begin{equation}
\label{eq:order n_j}
\left|Q(n) - \sum_{j=1}^{s+1}c_j\lambda^{k_j}\right|\le B
\end{equation}
satisfies 
\begin{equation}
\label{eq:bound 4}
\left\{n\in \cS_0:~ n\le N\right\} \ll  \(\log N\)^{s+1}. 
\end{equation}

Let $\Delta \in\N$ be sufficiently large (but depending only on $|\lambda|$, which is larger than $1$, and also depending on the absolute values of the $c_1, \ldots, c_{s+1}$) such that we have
\begin{equation}
\label{eq:largest power of p}
\frac{|c_{s+1}|}{2}\cdot |\lambda|^{k_{s+1}} \le \left|\sum_{j=1}^{s + 1}c_j\lambda^{k_j}\right|,
\end{equation}
for all integers $k_{s+1} \ge \cdots \ge k_1 \ge 0$  satisfying the 
inequality~\eqref{eq:order n_j} along with the inequality $k_{s+1}-k_s\ge \Delta $.

Now, we let $\cU$ be the subset of $\cS_0$ consisting of integers $n\in\N$ for which one can find integers $k_j$ satisfying~\eqref{eq:order n_j} and in addition, $k_{s+1}-k_s<\Delta $. Then the existence of such a solution tuple $(k_1,\ldots, k_{s+1})$ for each $n\in \cU$ means that 
$$
\left|Q(n) -\left(c_1\lambda^{k_1}+\cdots +c_{s-1}\lambda^{k_{s-1}} + \left(c_s+c_{s+1}\lambda^{k_{s+1}-k_s}\right)\lambda^{k_s}\right)\right|\le B.
$$
Because $k_{s+1}-k_s\in\{0,\ldots, \Delta -1\}$, applying the induction hypothesis for each of the possible $\Delta $ values of $k_{s+1}-k_s$, we obtain the desired conclusion regarding the asymptotic growth given by~\eqref{eq:bound 4} (furthermore, we actually get that the exponent from the right-hand side of the inequality~\eqref{eq:bound 4} is $s=m-1$ not $s+1=m$).  

On the other hand, for each $n\in \cS_0\setminus \cU$ satisfying $n\ge N_0$, we know that there must exist some tuple of nonnegative integers $(k_1,\ldots, k_{s+1})$ satisfying~\eqref{eq:order n_j} and in addition, $k_{s+1}-k_s\ge \Delta $. Then using both~\eqref{eq:Q growth} 
and~\eqref{eq:largest power of p}, we  get 
$$
\frac{|c_{s+1}|}{2}\cdot |\lambda|^{k_{s+1}} - B \le \left|c_0+\sum_{j=1}^{s+1}c_j \lambda^{k_j}\right| - B \le  |Q(n)| \le 2|a_d|\cdot n^d,
$$
which implies that 
\begin{equation}
\label{eq:k s 1}
k_{s+1} \le b_0 \(1+\log n\) 
\end{equation} 
for some positive real number $b_0 $ depending only on $B$, $|\lambda|$, $d$,  $|a_d|$ and $|c_{s+1}|$. 

So, let $N$ be an integer larger than $N_0$; then for each integer $N_0\le n\le N$ contained in $\cS_0\setminus \cU$, we know there exists an $(s+1)$-tuple of integers $k_i$ satisfying~\eqref{eq:inequalities} and~\eqref{eq:order n_j}. Combining the fact that  $1\le k_i\le k_{s+1}$ with the inequality~\eqref{eq:k s 1}, we get that there are at most $\(b_0 \(1+\log N\)\)^{s+1}$ tuples $(k_1,\ldots,k_{s+1})$ for which we could find some $n\in \cS_0\setminus \cU$ satisfying the inequality $N_0\le n\le N$. However, since $n\ge N_0$, then the inequality~\eqref{eq:spaced-out} yields that for any such $(s+1)$-tuple of integers $k_i$, there are \emph{at most} $N_0$ integers $n\in \left(\cS_0\setminus \cU\right)\cap [N_0,N]$ satisfying~\eqref{eq:order n_j} with respect to the tuple $(k_1,\ldots, k_{s+1})$. Hence,  we get the inequality
$$
\#\left\{ n\le N\colon~ n\in \cS_0\setminus \cU\right\} \le 
N_0\cdot \left(1+\left(b_0 \cdot  \(1+\log N\)\right)^{s+1}\right).
$$
for each positive integer $N\ge N_0$. This concludes our proof of  
Theorem~\ref{thm:Qn lambda B}.

\section{Construction and properties of the sieving set of primes}   
\label{sec:prelim}

\subsection{Multiplicative orders}

Let  $\tau_{\ell}(g)$ denote
the multiplicative order of an integer $g\ge 2 $ modulo a prime $\ell$, that is,
the smallest positive integer $\tau$ for which $g^\tau\equiv 1\bmod \ell$.

Let $\alpha$  be a fixed real number such that
\begin{equation}
\label{eq:good alpha}
\# \left\{\ell \le z:\ell~\text{is prime and~}
P(\ell-1)\ge \ell^{\alpha}\right\} \gg \frac{z}{\log z}
\end{equation}
for all sufficiently large $z$, where $P(k)$ denotes the largest
prime divisor of an integer $k\ge 2$, and the implied constant
depends only on $\alpha$.

We recall the following well known result which follows from 
the divisibility $ \tau_\ell(g)\mid \ell-1$ (provided $\gcd(g,\ell)=1$)
and  the bound 
$$
\# \left\{\ell \le z:\ell~\text{is prime,~}P(\ell-1) > \ell^{1/2} \right\} = (1+o(1)) \frac{z}{\log z} 
$$
as $z \to \infty$, which easily follows from a stronger result of
 Erd\H os and  Murty~\cite[Theorem~3]{ErdMur}.
 Details can be found in the work of 
 Kurlberg and  Pomerance~\cite[Lemma~20]{KurPom}.

\begin{lemma}
\label{lem:Mult Ord} For any fixed  $\alpha\ge 1/2$ satisfying~\eqref{eq:good alpha} 
and  any fixed integer $g\ge 1$ we have
$$
\# \left\{\ell \le z:\ell~\text{is prime,~} \tau_\ell(g) \ge \ell^\alpha\right\} \gg \frac{z}{\log z} 
$$ 
as $z \to \infty$.
\end{lemma}
 
For an integer $s\ge 1$ we denote by $\nu_2(s)$ the $2$-adic order of $s$, that is, 
the largest power $\nu$ such that $2^\nu \mid s$. 
 
\begin{lemma}
\label{lem:Mult Ord Dense} For any fixed  $\alpha\ge 1/2$ satisfying~\eqref{eq:good alpha} 
and  any fixed integer $g\ge 1$ there are some absolute 
constants $C_1, C_2  >0$,  such 
that for every sufficiently large real number $z>1$, there exist some integer $u_0$
and  a set $\cL_z \subseteq [z,C_1z]$ of primes of cardinality
$$
\#   \cL_z \ge  \frac{C_2z}{  \log z \log \log z}
$$
such that  for every $\ell \in  \cL_z$ we have
$$
  P(\ell-1) \ge z^\alpha, \qquad P(\ell-1) \mid \tau_\ell(g),  \qquad \nu_2\( \tau_\ell(g)\) = u_0. 
$$
\end{lemma}

\begin{proof} 
Lemma~\ref{lem:Mult Ord} obviously implies that for some absolute constants $C_1, C_3$  there are at least 
 $C_3z/\log z$ primes $\ell\in[z,C_1z]$ satisfying only the first two conditions, see also~\cite[Lemma~5.1]{LuSh}.  
 Let $\overline \cL_z$  be this set. 
Trivially, there are at at most $z/2^{v_0}$ primes $\ell$ with $\nu_2(\tau_\ell(g))\ge v_0$.
Hence  taking a sufficiently large  $C_4$,  and   $v_0 = \fl{C_4 \log \log z}$, 
we  see that if we remove these primes  from 
$\overline \cL_z$ we obtain the set of $\widetilde \cL_z\subseteq \overline \cL_z$ 
of cardinality 
$$
\# \widetilde \cL_z \ge \# \overline \cL_z - z/2^{v_0} \ge 0.5 \# \overline \cL_z \ge 0.5 C_3z/\log z.
$$
Since obviously $\nu_2\( \tau_\ell(g)\) \le  \nu_2\(\ell-1\) \le v_0$, making a majority 
decision we can find a set of $\cL_z$ of cardinality 
$$
\#   \cL_z \ge \frac{\# \widetilde \cL_z }{v_0} \ge \frac{0.5 C_3z}{v_0 \log z}
$$
with $\nu_2\( \tau_\ell(g)\) = u_0$ for some fixed $u_0 \le v_0$  for every $\ell \in  \cL_z$. 
Taking $C_2 = 0.5 C_3/C_4$ we conclude the proof. 
\end{proof} 

We note that the {\it Brun-Titchmarsh\/} theorem (see~\cite[Theorem~6.6]{IwKow}) can be 
used to remove $\log \log z$ in the bound on $\#   \cL_z$ of  Lemma~\ref{lem:Mult Ord Dense}. 
However in our final result we do not try to optimise terms of this order, so we ignore this 
and similar potential improvements.

\subsection{Sieving set $\cL_z$} 
\label{sec:Set Lz} 
We see that using a result of Baker and Harman~\cite{BaHa} one can take
\begin{equation}
\label{eq:BH bound} 
 \alpha =0.677,
\end{equation}
in Lemmas~\ref{lem:Mult Ord} and~\ref{lem:Mult Ord Dense}. 

From now on, for any positive real number $z$, we fix a set $\cL_z$ satisfying the conclusion 
of Lemma~\ref{lem:Mult Ord Dense} with $\alpha$ given by~\eqref{eq:BH bound}.  

\subsection{Bounds of some arithmetic sums} 
\label{eq:prime div}
For an integer $K$ we consider the set  
\begin{equation}
\label{eq:set K}
 \cK =\cK_m(K) 
\end{equation}
where 
\begin{equation}
\label{eq:set KmK}
\cK_m(K)=  \{0, \ldots, K\}^m , 
\end{equation}
and for  $\vk =(k_1, \ldots, k_m) \in \cK$ we define 
\begin{equation}
\label{eq:Fk}
 F(\vk) = \sum_{i=1}^m c_i g^{k_i}. 
\end{equation}

 For a real $z\ge 2$ let  $\omega_z\(n\)$ be the number of distinct prime factors $\ell \in \cL_z$ of $n$.

\begin{lemma}
\label{lem:omega-z} Let an integer $K$ and a real $z$  be sufficiently large.
For  $\cK$ and  $F(\vk)$ 
as in~\eqref{eq:set K} and~\eqref{eq:Fk}, respectively, we have   
$$
\sum_{\vk \in \cK} \omega_z\(F(\vk)\) \ll  \(K^m z^{-\alpha}+ K^{m-1}\)  \# \cL_z .
$$
\end{lemma}

\begin{proof} We have
$$
\sum_{\vk \in \cK} \omega_z\(F(\vk)\)   \ll    \sum_{\vk \in \cK} 
\sum_{\substack{\ell \in \cL_z\\ \ell  \mid  F(\vk)}} 1 = \sum_{\ell
\in \cL_z} \sum_{\substack{\vk \in \cK\\ \ell  \mid  F(\vk)}}  1.
$$
Clearly the last sum can be estimated as
\begin{align*}
\sum_{\substack{\vk \in \cK\\ \ell  \mid  F(\vk)}}  1 
& \le (K+1)^{m-1} \(\frac{K+1}{\tau_\ell(g)} + 1\)\\
& \ll K^m \ell^{-\alpha} + K^{m-1} \le  K^mz^{-\alpha} + K^{m-1}, 
\end{align*}
and the result follows. 
\end{proof}

\begin{remark} The proof of Lemma~\ref{lem:omega-z} appeals to essentially trivial bound 
$ O\( K^{m-1} \(K/\tau_\ell(g) + 1\)\)$ on the number of solution to the congruence $F(\vk)\equiv 0 \pmod \ell$, 
$\vk \in \cK$. Using bounds of exponential sums one can obtain a  better bound, which however does not improve our 
final result (see also our Appendix). 
\end{remark}

For a real $\kappa$ we define the sums 
$$
  D_\kappa =  \sum_{\substack{\ell,r\in\cL_z\\P(\ell-1)\ne P(r-1)}}  \gcd\(\ell-1, r-1\)^\kappa.
$$
\begin{lemma}
\label{lem:Sum gcd} Let  a real $z$  be sufficiently large.
Then for $\kappa\ge 1$ we have
$$
D_\kappa \le  z^{ \kappa +\alpha - \alpha \kappa+1 +o(1) } . 
$$
\end{lemma}

\begin{proof} Clearly for each pair of primes $(\ell,r)$ in the sum $  D_\kappa$ we have 
$ \gcd\(\ell-1, r-1\)\le  (\ell-1)/P(\ell-1) \le H$  for some integer $H \ll z^{1-\alpha}$. 
Hence 
$$
D_\kappa  \le \sum_{h=1}^H h^{\kappa} \sum_{\substack{\ell,r\in\cL_z\\P(\ell-1)\ne P(r-1)\\  \gcd\(\ell-1, r-1\) = h}} 1. 
$$
We estimate the inner sum trivially as $O\((z/h)^2\)$ and derive
$$
D_\kappa  \le  z^2 \sum_{h=1}^H h^{\kappa-2}  \le z^{2+o(1)} H^{\kappa-1} \le z^{2+(\kappa-1)(1-\alpha)+o(1) } 
,
$$
and the desired result follows. 
\end{proof}

\section{Bounds of character sums}

 \subsection{Complete character sums with diagonal forms over finite fields} 
Let $q$ be an odd prime power and let $\F_q$ be the finite field of $q$ elements. 
We note that for the purpose of proving Theorem~\ref{thm:Qn g}, we only need to estimate the sums of this section
over a prime finite field. However, since our proofs work over arbitrary finite fields, we present them in this more general setting with the hope they would be of independent interest.

We let $m\ge 1$ and $d\ge 2$ be  integers with $d$ coprime with $q$.

Let $\cX$ denote the set 
of multiplicative characters of $\F_q^*$ and let $\cX^* = \cX\setminus \{\chi_0\}$ be the 
set of non-principal characters, we refer to~\cite[Chapter~3]{IwKow} for a background  on characters. 
 We also denote by $\eta \in \cX^*$ the quadratic characters (that is $\eta^2 = \chi_0$). 

We recall that the implied constant may depend on $m$ (but not on $d$,  $q$ and other parameters). 

We start with `pure' bounds of sums of quadratic characters.

We note that in our next result we have an additional condition of $d$ being an even 
integer. 

\begin{lemma}
\label{lem:S pure}
 Assume that   the integer  $d\ge 2$ satisfies $\gcd(d, q)=1$ and  is even. 
Let $a_1,\ldots,a_m\in\F_q^*$.   Then for 
$$
S =  \sum_{x_1,\ldots,x_m\in\F_q} \eta\(a_1x_1^d+\cdots+a_mx_m^d\)
$$
we have 
$$
\left| S \right|   \le d^{m-1}(q-1)q^{(m-1)/2}.
$$
\end{lemma}

\begin{proof}
The proof follows by induction on $m$.  For $m=1$, since $d$ is even,  the sum becomes
$$
\left| \sum_{x_1\in\F_q} \eta\(a_1x_1^d\) \right|= q-1.
$$ 
We assume the bound true for $m-1$ and we prove it for $m$. We have
\begin{align*}
S=\sum_{x_m\in\F_q^*}&\sum_{x_1,\ldots,x_{m-1}\in\F_q} \eta\(a_1x_1^d+\cdots+a_mx_m^d\)\\
&\qquad\qquad+\sum_{x_1,\ldots,x_{m-1}\in\F_q} \eta\(a_1x_1^d+\cdots+a_{m-1}x_{m-1}^d\).
\end{align*}
By the induction hypothesis, the second sum in the above is  bounded by $d^{m-2}(q-1)q^{(m-2)/2}$. Hence,  we have
\begin{equation}
\label{eq:S S*}
\left| S \right| \le \left| S^* \right| +  d^{m-2}(q-1)q^{(m-2)/2}, 
\end{equation}
where 
$$
S^*=\sum_{x_m\in\F_q^*}\sum_{x_1,\ldots,x_{m-1}\in\F_q}\eta\(a_1x_1^d+\cdots+a_mx_m^d\),
$$
to which we apply~\cite[Theorem~2.1]{Ka02}. Indeed, since $x_m\ne 0$ in $S^*$, we make the transformation $x_i\to x_ix_m$, $i=1,\ldots,m-1$, 
which does not change the sum. Moreover, since again $d$ is even and $\eta(x_m^d)=1$, we obtain
\begin{equation}
\begin{split}
\label{eq:S*}
S^*&=\sum_{x_m\in\F_q^*} \sum_{x_1,\ldots,x_{m-1}\in\F_q}\eta\(a_1x_1^d+\cdots+a_{m-1}x_{m-1}^d+a_m\)\\
& =(q-1)\sum_{x_1,\ldots,x_{m-1}\in\F_q}\eta\(a_1x_1^d+\cdots+a_{m-1}x_{m-1}^d+a_m\).
\end{split}
\end{equation}

Let now $$F(X_1,\ldots,X_{m-1})=a_1X_1^d+\cdots+a_{m-1}X_{m-1}^d+a_m\in \F_q[X_1,\ldots,X_{m-1}].$$ We note that the equation $F(X_1,\ldots,X_{m-1})=0$ defines a smooth hypersurface in the affine space  $\A^{m-1}(\F_q)$. Indeed, considering the partial derivatives of $F$ with respect to each variable $X_i$, we obtain that the only possible singular point would be $(0,\ldots,0)$. However, since $a_m\ne 0$, this point does not belong to the hypersurface $F(X_1,\ldots,X_{m-1})=0$.

Similarly, the equation given by the leading homogenous part of $F$, $a_1X_1^d+\cdots+a_{m-1}X_{m-1}^d=0$, defines a smooth hypersurface in 
the projective space $\P^{m-2}(\F_q)$.

Applying now~\cite[Theorem 2.1]{Ka02}, we conclude from~\eqref{eq:S*} that
\begin{equation}
\label{eq:S* bound}
 \left| S^* \right|  \le (d-1)^{m-1}  (q-1) q^{(m-1)/2}.
\end{equation}
Substituting~\eqref{eq:S* bound} in~\eqref{eq:S S*}, we obtain
\begin{align*}
|S| & \le  (d-1)^{m-1}  (q-1) q^{(m-1)/2} +  d^{m-2}(q-1)q^{(m-2)/2} \\ 
& \le d^{m-2}(q-1)q^{(m-2)/2}  \((d-1)q^{\frac{1}{2}}+1\).
\end{align*}
Since $(d-1)q^{\frac{1}{2}}+1< dq^{1/2}$, we conclude the proof.
\end{proof}

Next we need the following bound on multidimensional sum of quadratic characters, twisted by 
arbitrary characters. In the next result we do not use that $d$ is even.

\begin{lemma}
\label{lem:S twist} Assume that the integer $d\ge 2$ satisfies $\gcd(d, q)=1$.
Let $a_1,\ldots,a_m\in\F_q^*$.   Then for any $\chi_1, \ldots, \chi_m \in \cX$ we have 
$$
\sum_{x_1,\ldots,x_m\in\F_q} \eta\(a_1x_1^d+\cdots+a_mx_m^d\)\chi_1\(x_1\) \ldots\chi_m\(x_m\)  \ll d^{m}q^{(m+1)/2}.
$$
\end{lemma}

\begin{proof} First we note that if each $\chi_i$ is equal to the principal character, then the result follows from Lemma~\ref{lem:S pure}. So, from now on, we assume that not all of the characters $\chi_i$ are equal to the principal character.

We have 
\begin{equation}
\label{eq:S1S0}
\sum_{x_1,\ldots,x_m\in\F_q} \eta\(a_1x_1^d+\cdots+a_mx_m^d\)\chi_1\(x_1\) \ldots\chi_m\(x_m\)   =   S_1 - S_0, 
\end{equation}
where
$$
S_0= \sum_{x_1,\ldots,x_m\in\F_q}\chi_1\(x_1\) \ldots\chi_m\(x_m\)  $$
and 
$$
S_1 = \sum_{y \in \F_q^*}\sum_{\substack{x_1,\ldots,x_m\in\F_q \\a_1x_1^d+\cdots+a_mx_m^d=y^2}}  \chi_1\(x_1\) \ldots\chi_m\(x_m\)  .
$$ 
Indeed,  we observe that each vector $(x_1,\ldots,x_m)\in\F_q^m$ contributes $2  \chi_1(x_1) \ldots\chi_m(x_m)$ to the sum $S_1$.
It is also easy to see that $S_0$ vanishes unless $\chi_1 =   \ldots =  \chi_m = \chi_0$; therefore, due to our assumption from above, we get that $S_0=0$.  We now fix a nontrivial additive character  $\psi$ of $\F_q$. 
By the orthogonality relation, 
$$
 \frac{1}{q}  \sum_{\lambda \in \F_q} \psi\(\lambda u\)=
\begin{cases} 1, & \text{if}\ u = 0,\\
0, & \text{if}\ u \in \F_q^*,
\end{cases}
$$
see~\cite[Section~3.1]{IwKow}. 
Hence we write 
\begin{align*}
S_1 &=  \sum_{x_1,\ldots,x_m\in\F_q} \sum_{ y \in \F_q^*} \\
&\qquad\qquad\quad \frac{1}{q}   \sum_{\lambda \in \F_q} \psi\(\lambda\(a_1x_1^d+\cdots+a_mx_m^d-y^2\)\) \chi_1(x_1) \ldots\chi_m(x_m)   \\ 
 &=  \frac{1}{q}    \sum_{\lambda \in \F_q} \sum_{ y \in \F_q^*} \psi\(-\lambda y^2\)  \\ 
 &\qquad\qquad\quad\sum_{x_1,\ldots,x_m\in\F_q}  \psi\(\lambda\(a_1x_1^d+\cdots+a_mx_m^d\)\)  \chi_1(x_1) \ldots\chi_m(x_m)  \\
  &=  \frac{1}{q}    \sum_{\lambda \in \F_q} \sum_{ y \in \F_q^*} \psi\(-\lambda y^2\)   \prod_{i=1}^m\sum_{x_i\in\F_q}  \psi\(\lambda a_ix_i^d\)  \chi_i\(x_i\)  .
\end{align*}
The contribution from the terms corresponding to $\lambda = 0 $ is obviously equal to  $\frac{q-1}{q}\cdot S_0=0$ since $S_0=0$ (because not all of the characters $\chi_i$ are equal to the principal character). 
Hence 
\begin{equation}
\label{eq:S1S0W}
S_1  =    W, 
 \end{equation}
 where 
 $$
 W =  \frac{1}{q}    \sum_{\lambda \in \F_q^*} \sum_{ y \in \F_q^*} \psi\(-\lambda y^2\)   \prod_{i=1}^m\sum_{x_i\in\F_q}  \psi\(\lambda a_ix_i^d\)  \chi_i\(x_i\) .
 $$
Now the sum over $y$ differs from the classical Gauss sums by only one term corresponding to $y=0$,  and so we 
have 
\begin{equation}
\label{eq:Gauss}
 \sum_{ y \in \F_q^*} \psi\(-\lambda y^2\) \ll q^{1/2},
 \end{equation}
see~\cite[Theorem~3.4]{IwKow}.
For the remaining sums, using that $\lambda a_i \in \F_q^*$ we apply the Weil bound~\cite[Appendix~5, Example~12]{Weil}
 of mixed sums of additive and multiplicative characters which implies
\begin{equation}
\label{eq:Weil}
\sum_{x_i\in\F_q}  \psi\(\lambda a_ix_i^d\)  \chi_i\(x_i\)   \ll dq^{1/2}, 
 \end{equation}
 see also~\cite[Chapter~6, Theorem~3]{Li}. 
 Therefore, the bounds~\eqref{eq:Gauss} and~\eqref{eq:Weil}, combined together yield
$$
W \ll d^{m}q^{(m+1)/2}.
$$
and together with~\eqref{eq:S1S0} and~\eqref{eq:S1S0W} we conclude the proof. 
\end{proof}

We remark that some, or all, of the characters $\chi_1, \ldots, \chi_m \in \cX$
can be principal and that the implied constant from the conclusion of Lemma~\ref{lem:S twist} depends only on $m$.

\subsection{Incomplete character sums with exponential functions} 

We now extend the definition of $\tau_\ell(g)$ to orders modulo any 
composite moduli $q$ with $\gcd(g,q)=1$.  We also use $(u/q)$ to denote the
{\it Jacobi symbol\/} modulo  an odd $q$. 

Here we need to obtain multidimensional analogues of the result on character sums from~\cite[Section~3]{BaSh}.
Although this does not require new ideas and can be achieved at the cost of merely typographical changes 
we present some short proofs of these results. 

As usual, we write $\e(t)= \exp(2 \pi i t)$ for all $t\in\R$.

We use the following variant of the result of~\cite[Lemma~3.1]{BaSh},  which 
in turn is based on some ideas of Korobov~\cite[Theorem~3]{Kor}.

\begin{lemma}
\label{lem:Prod Form} Let $a_1, b_1 \ldots, a_m, b_m \in \Z$ and let $ \vartheta\in\Z$ with $ \vartheta\ge 2$.
Let   $\ell$ and $r$  be distinct primes with 
$$
t_\ell = \tau_\ell(\vartheta), \qquad t_r = \tau_r( \vartheta), \qquad t = \tau_{\ell r}( \vartheta)
$$
and such that
$$
\gcd\(\ell r, a_1  \ldots a_m \vartheta \) = \gcd\(t_\ell,  t_r\) = 1.
$$ 
We define integers $b_{i,\ell}$ and $b_{i,r}$ by the conditions
$$
b_{i,\ell} t_r+ b_{i, r}   t_\ell  \equiv b_i \pmod {t},\qquad 
0 \le b_{i,\ell} < t_\ell, \quad 0 \le b_{i,r} <  t_r, 
$$
for $i =1, \ldots, m$. Then,
for 
$$
S =  \sum_{k_1, \ldots, k_m=1}^{t}\(\frac{a_1 \vartheta^{k_1} + \ldots + a_m \vartheta^{k_m}}{\ell  r}\)  
\e\(\frac{b_1k_1+ \ldots+ b_mk_m}{t}\)
$$
we have 
$$
S = S_\ell S_r
$$
where 
\begin{align*}
& S_\ell = 
 \sum_{x_1, \ldots, x_m=1}^{t_\ell}\(\frac{a_1 \vartheta^{x_1} + \ldots + a_m \vartheta^{x_m}}{\ell}\) 
\e\(\frac{b_{1,\ell} x_1+ \ldots+ b_{m,\ell} x_m}{t_\ell}\), \\
& S_r =  \sum_{y_1, \ldots, y_m=1}^{t_r}\(\frac{a_1 \vartheta^{y_1} + \ldots + a_m \vartheta^{y_m}}{r}\) 
\e\(\frac{b_{1,r} y_1+ \ldots+ b_{m,r} y_m}{t_r}\).
\end{align*} 
\end{lemma}

\begin{proof} As in  the proof of~\cite[Lemma~3.1]{BaSh}, 
using that    $\gcd(t_\ell,  t_r)=1$,  we see that  
the integers 
$$
x t_r+ yt_\ell,\qquad  0 \le x <t_\ell,\qquad \ 0 \le y <  t_r, 
$$
run through the complete residue system modulo 
$$
t =t_\ell t_r.
$$ 
Moreover,
\begin{equation}
\label{eq:Congr  mod ell r}
 \vartheta^{x t_r + y t_\ell} \equiv \vartheta^{x t_r} \pmod \ell,\qquad  
 \vartheta^{x t_r + y t_\ell} \equiv \vartheta^{y t_\ell} \pmod r,
\end{equation}
and
\begin{equation}
\label{eq:exp split}
 \e(b (x t_r + y t_\ell)/t)= 
  \e (bx/t_\ell) \e(b y/t_r).
\end{equation}
Hence,
\begin{equation}
\begin{split}
\label{eq:S split} 
S &   = \sum_{x_1, \ldots, x_m=1}^{t_\ell}  \sum_{y_1, \ldots, y_m=1}^{t_{r}}
\(\frac{a_1 \vartheta^{x_1 t_r+ y_1t_\ell}+ \ldots + a_m \vartheta^{x_m t_r+ y_mt_\ell}}{\ell  r}\)  \\
& \qquad \qquad \quad   \e\(\frac{b_1\(x_1 t_r+ y_1t_\ell\)+ \ldots+ b_m\(x_m t_r+ y_mt_\ell\)}{t}\).
\end{split}
\end{equation}

Using the multiplicativity of the Jacobi symbol, and recalling the congruences~\eqref{eq:Congr  mod ell r}, we derive 
\begin{equation}
\begin{split}
\label{eq:Mult Jac}
& \(\frac{a_1 \vartheta^{x_1 t_r+ y_1 t_\ell} + \ldots + a_m \vartheta^{x_m t_r + y_m t_\ell}}{\ell  r}\)  \\
& \qquad \qquad  =  \(\frac{a_1 \vartheta^{x_1 t_r+ y_1 t_\ell} + \ldots + a_m \vartheta^{x_m t_r + y_m t_\ell}}{\ell }\)  \\
& \qquad \qquad \qquad  \qquad \(\frac{a_1 \vartheta^{x_1 t_r+ y_1 t_\ell} + \ldots + a_m \vartheta^{x_m t_r + y_m t_\ell}}{r}\)  \\
& \qquad \qquad  =  \(\frac{a_1 \vartheta^{x_1 t_r} + \ldots + a_m \vartheta^{x_m t_r }}{\ell }\)  \\
& \qquad \qquad \qquad \qquad  \(\frac{a_1 \vartheta^{y_1 t_\ell} + \ldots + a_m \vartheta^{y_m t_\ell}}{r}\) .
\end{split}
\end{equation}
Furthermore, by~\eqref{eq:exp split} we have
\begin{equation}
\begin{split}
\label{eq:Mult  Exp}
&   \e\(\frac{b_1\(x_1 t_r+ y_1t_\ell\)+ \ldots+ b_m\(x_m t_r+ y_m t_\ell\)}{t}\)\\
& \qquad \quad  =   \e\(\frac{b_1x_1  + \ldots+ b_m x_m  }{t_\ell}\)   \e\(\frac{b_1  y_1 + \ldots+ b_m t_m   }{t_r}\). 
\end{split}
\end{equation}
Using~\eqref{eq:Mult Jac} and~\eqref{eq:Mult Exp} in~\eqref{eq:S split}, we see that the sum $S$ can be decomposed 
into a product of two sums as follows
\begin{align*}
S &   = \sum_{x_1, \ldots, x_m=1}^{t_\ell}   \(\frac{a_1 \vartheta^{x_1 t_r} + \ldots + a_m \vartheta^{x_m t_r }}{\ell }\)   \e\(\frac{b_1x_1  + \ldots+ b_m x_m  }{t_\ell}\) \\
& \qquad \quad \sum_{y_1, \ldots, y_m=1}^{t_{r}}  \(\frac{a_1 \vartheta^{y_1 t_\ell} + \ldots + a_m \vartheta^{y_m t_\ell}}{r}\)  \e\(\frac{b_1  y_1 + \ldots+ b_m t_m   }{t_r}\).
\end{align*}
We now replace $x_i$ with $x_i  t_r^{-1} \pmod {t_\ell}$ 
and $y_i$ with $y_i  t_\ell^{-1} \pmod {t_r}$,
and take into account that
$$b_i  t_r^{-1} \equiv b_{i,\ell} \pmod {t_\ell}
\mand
b_i t_\ell^{-1} \equiv b_{i,r} \pmod {t_r},
$$
for $i =1, \ldots, m$. This concludes the proof. 
\end{proof}

Next we estimate the sums $S_\ell$ and $S_r$ which appear in Lemma~\ref{lem:Prod Form}. 
Namely we now establish an analogue of~\cite[Lemma~3.2]{BaSh}. 

\begin{lemma}
\label{lem:Char ell}
Let $a_1, b_1 \ldots, a_m, b_m \in \Z$ and let $ \vartheta\in\Z$ with $ \vartheta\ge 2$.
Let   $\ell$   be  a  prime with 
$$
t_\ell = \tau_\ell(\vartheta)
$$
and such that 
$$
\gcd\(\ell, a_1  \ldots a_m \vartheta \)  = 1 \mand \gcd(t_\ell, 2) = 1.
$$ 
Then for 
\begin{align*}
& S_\ell = 
 \sum_{x_1, \ldots, x_m=1}^{t_\ell}\(\frac{a_1 \vartheta^{x_1} + \ldots + a_m \vartheta^{x_m}}{\ell}\) 
\e\(\frac{b_{1} x_1+ \ldots+ b_{m} x_m}{t_\ell}\)
\end{align*} 
we have 
$$
S_\ell \ll  \begin{cases} \ell ^{(m+1)/2}, & \text{for arbitrary}\  b_1, \ldots, b_m,\\
t_\ell  \ell^{(m-1)/2} , &  \text{for}\  b_1=  \ldots = b_m = 0.
\end{cases}
$$
\end{lemma}

\begin{proof} Denoting $d = (\ell-1)/t_\ell$, 
we can write $\vartheta = \rho^d$ with \emph{some} primitive root $\rho$ modulo $\ell$. 
Then,
\begin{equation}
\begin{split}
\label{eq:S DiagForm} 
S_\ell
&= \frac{1}{d^m} \sum_{x_1, \ldots, x_m=1}^{\ell-1}\(\frac{a_1 \rho^{dx_1} + \ldots + a_m \rho^{dx_m}}{\ell}\) \\
& \qquad  \qquad  \qquad  \qquad  \qquad \e\(\frac{d\(b_{1}  x_1+ \ldots+ b_{m} x_m\)}{\ell-1}\)\\
&=\frac{1}{d^m}  \sum_{w_1, \ldots, w_m =1}^{\ell-1}\(\frac{a_1 w_1^d + \ldots + a_m  w_m^d}{\ell}\) 
\chi_1(w_1)\ldots \chi_m(w_m), 
\end{split}
\end{equation}
where for $w \in \F_\ell$ we define $\chi_i$    by
$$
\chi_i(w)=\e\(b_i dx/(\ell-1)\), \qquad i=1, \ldots, m,
$$
where $x$ is any integer for which $w\equiv\rho ^x\pmod \ell$.

As in the proof of~\cite[Lemma~3.2]{BaSh} we observe  $\chi_i$ is a multiplicative character of $\F_\ell$  for 
each $i=1, \ldots, m$. Recalling Lemma~\ref{lem:S twist}, we derive from~\eqref{eq:S DiagForm} 
$$
S_\ell \ll \frac{1}{d^m}  d^{m}\ell^{(m+1)/2} = \ell^{(m+1)/2}, 
$$
which establishes  the desired bound for arbitrary $b_1, \ldots, b_m \in \Z$. 

For $b_1=  \ldots = b_m = 0$ we observe that since $t_\ell$ is odd, $d$ is even and hence 
we can use Lemma~\ref{lem:S pure} instead of Lemma~\ref{lem:S twist}. 
Thus in this case~\eqref{eq:S DiagForm} implies 
$$
S_\ell \ll \frac{1}{d^m}  d^{m-1}\ell^{(m+1)/2} = \frac{1}{d}  \ell^{(m+1)/2} \ll t_\ell  \ell^{(m-1)/2},
$$
which concludes the proof. 
\end{proof}

Lemmas~\ref{lem:Prod Form} and~\ref{lem:Char ell} combined together  imply the following bound.

\begin{cor} 
\label{cor:S compl} Let $a_1, b_1 \ldots, a_m, b_m \in \Z$ and let $ \vartheta\in\Z$ with $ \vartheta\ge 2$.
Let   $\ell$ and $r$  be distinct primes with 
$$
t_\ell = \tau_\ell(\vartheta), \qquad t_r = \tau_r( \vartheta), \qquad t = \tau_{\ell r}( \vartheta)
$$
and such that
$$
\gcd\(\ell r, a_1  \ldots a_m \vartheta \) = \gcd\(t_\ell,  t_r\) =   \gcd(t_\ell t_r, 2) = 1.
$$ 
Then,
for 
$$
S =  \sum_{k_1, \ldots, k_m=1}^{t}\(\frac{a_1 \vartheta^{k_1} + \ldots + a_m \vartheta^{k_m}}{\ell  r}\)  
\e\(\frac{b_1k_1+ \ldots+ b_mk_m}{t}\)
$$
we have 
$$
S \ll \begin{cases}  (\ell r)^{(m+1)/2}, & \text{for arbitrary}\  b_1, \ldots, b_m,\\
t  (\ell r)^{(m-1)/2}, &  \text{for}\  b_1=  \ldots = b_m = 0.
\end{cases}
$$
\end{cor}

Clearly in Lemma~\ref{lem:Char ell} and Corollary~\ref{cor:S compl} the parity condition on 
multiplicative orders is important only in the case where $b_1=  \ldots = b_m = 0$, 
as only these parts appeal to Lemma~\ref{lem:S pure} (which required $d$ to be even). 

Combining  Corollary~\ref{cor:S compl} with the completing method, see~\cite[Section~12.2]{IwKow}, 
we derive an analogue of~\cite[Lemma~3.4]{BaSh}, which is our 
main technical tool.

\begin{lemma}
\label{lem:Char Main} 
Let $a_1, \ldots, a_m \in \Z$ and let $ \vartheta\in\Z$ with $ \vartheta\ge 2$.
Let   $\ell$ and $r$  be distinct primes with 
$$
\gcd\(\ell r, a_1  \ldots a_m \vartheta \) = \gcd\(\tau_\ell( \vartheta), \tau_r( \vartheta)\) 
= \gcd(\tau_\ell( \vartheta)\tau_r( \vartheta), 2)= 1.
$$ 
Then, for any integers   $L_1, \ldots L_m \ge 1$,  we have
\begin{align*}
&\sum_{k_1 =1}^{L_1}\ldots \sum_{k_m =1}^{L_m}\(\frac{a_1\vartheta^{k_1}  + \ldots + a_m\vartheta^{k_m}}{\ell   r}\) \\
&\qquad \qquad   \ll  L_1 \ldots L_m t^{-m+1} (\ell r)^{(m-1)/2} \\
&\qquad \qquad  \qquad \qquad  \quad  + \(L^{m-1}  t^{-m+1} +1\) (\ell r)^{(m+1)/2}\( \log (\ell r) \)^m,
\end{align*}
where 
$$
L = \max\{L_1, \ldots L_m\} \mand  t = \tau_{\ell r} ( \vartheta), 
$$ 
and the implied constant is absolute.  
\end{lemma}

\begin{proof} Clearly we can split the above sum into $\fl{L_1/t}\cdot \cdots \cdot \fl{L_m/t}$ complete sums, where each variable 
runs over the complete residue system modulo $t$ and into at most $O\((L/t)^{m-1}+1\)$ 
incomplete sums over a  complete residue system modulo $\ell r$. 

By Corollary~\ref{cor:S compl} each of these complete sums can be estimated as $O\(t  (\ell r)^{(m-1)/2} \)$, so they contribute
$O\(L_1 \ldots L_m t^{-m+1} (\ell r)^{(m-1)/2}\)$ in total. 

By the standard completing 
techniques,  see, for example,~\cite[Section~12.2]{IwKow}, we derive from Corollary~\ref{cor:S compl} that each incomplete 
sum can be estimated as $O\( (\ell r)^{(m+1)/2} \( \log (\ell r) \)^m\)$. 
Therefore, in total  they contribute
$O\(\(L^{m-1}  t^{-m+1} +1\)(\ell r)^{(m+1)/2}\( \log (\ell r) \)^m\)$. 

Combining both contributions together, we conclude the proof. 
\end{proof}

\section{Proof of Theorem~\ref{thm:Qn g}}

\subsection{Preliminary transformations}

 We can always assume that 
$$
k_m \ge \ldots \ge k_1.
$$

 We note that there is an integer  constant $h_0$ depending only on the initial data such that 
 if $k_m \ge k_{m-1} +h_0$ then 
 $$
n^2= \sum_{i=1}^m c_i g^{k_i} \ge 0.5 g^{k_m}
 $$ 
 and hence for some 
\begin{equation}
\label{eq:K log}
K \ll \log N
\end{equation}
we have 
\begin{equation}
\label{eq:max k}
k_m=\max\{k_1, \ldots, k_m \} \ll K.
\end{equation} 
On the other hand,  for $k_m < k_{m-1} +h_0$, writing 
$k_m = k_{m-1} + h$, $h =0, \ldots, h_0-1$ and
$$
n^2=\sum_{i=1}^m c_i g^{k_i} = \sum_{i=1}^{m-2} c_i g^{k_i} + (c_{m-1} + c_mg^h)  g^{k_{m-1}}
$$ 
by Theorem~\ref{thm:Qn lambda B} we obtain at most $O\((\log N)^{m-1}\)$ solutions $n\le N$.

 Hence we now estimate the number of solutions to~\eqref{eq:Qcig} with~\eqref{eq:max k}
 (for $K$ as in~\eqref{eq:K log}).

We recall the notation~\eqref{eq:set K} and~\eqref{eq:Fk}, and we write~\eqref{eq:Qcig}
as
\begin{equation}
\label{eq:sqr Fk}
n^2=   F(\vk).
\end{equation}
We also recall the definitions of the set $\cL_z$ in Section~\ref{sec:Set Lz}
and of $\omega_z(n)$ from Section~\ref{eq:prime div}.

To simplify the exposition everywhere below we replace logarithmic, and 
double logarithmic factors of $z$ with $z^{o(1)}$ (implicitly assuming that $z \to \infty$).
In particular we simply write 
\begin{equation}
\label{eq:Card L simple} 
\# \cL_z = z^{1+o(1)}.
\end{equation}
Since $z^{o(1)}$ also absorbs all implied constants, we use $\le$ instead 
of $\ll$ in the corresponding bounds. 

\subsection{Sieving} 
 Note that if $F(\vk)$ is a perfect square, then we always have
$$
\sum_{\ell \in \cL_z}\(\frac{ F(\vk)}{\ell}\)=\# \cL_z-\omega_z\(F(\vk)\).
$$
Hence 
$$
\# \cL_z \le \left| \sum_{\ell \in \cL_z}\(\frac{ F(\vk)}{\ell}\)\right| + \omega_z\(F(\vk)\). 
$$
Denote by $\cM$ the set of values of $\vk  \in  \cK$  satisfying~\eqref{eq:sqr Fk}
and let $M =\#\cM$ be its cardinality. 
Invoking Lemma~\ref{lem:omega-z}, we obtain 
\begin{align*}
M \# \cL_z & \le  \sum_{\vk \in \cM} \left| \sum_{\ell \in \cL_z}\(\frac{ F(\vk)}{\ell}\)\right| + \sum_{\vk \in \cM}  \omega_z\(F(\vk)\)\\
& \le  \sum_{\vk \in \cM} \left| \sum_{\ell \in \cL_z}\(\frac{ F(\vk)}{\ell}\)\right| + \sum_{\vk \in \cK}  \omega_z\(F(\vk)\)\\
& \ll    \sum_{\vk \in \cM} \left| \sum_{\ell \in \cL_z}\(\frac{ F(\vk)}{\ell}\)\right| + \(K^m z^{-\alpha}+ K^{m-1}\)\#\cL_z .
\end{align*} 
Therefore either 
\begin{equation}
\label{eq:ML Bound 1} 
M \# \cL_z   \ll   \sum_{\vk \in \cM} \left| \sum_{\ell \in \cL_z}\(\frac{ F(\vk)}{\ell}\)\right| 
\end{equation}
or 
\begin{equation}
\label{eq:ML Bound 2} 
M    \ll   K^m z^{-\alpha}+ K^{m-1}.
\end{equation} 

Assuming that~\eqref{eq:ML Bound 1} holds, by the Cauchy inequality
$$
\(M \# \cL_z \)^2  \le  M  \sum_{\vk \in \cM} \left| \sum_{\ell \in \cL_z}\(\frac{ F(\vk)}{\ell}\)\right|^2 
$$ 
and extending summation back to all $\vk \in \cK$ and using~\eqref{eq:Card L simple},  we obtain
\begin{equation}
\label{eq:M SqSieve1} 
M    \le z^{-2 +o(1)}   W, 
\end{equation}
where 
$$
W = \sum_{\vk \in \cK} \left| \sum_{\ell \in \cL_z}\(\frac{ F(\vk)}{\ell}\)\right|^2 
= \sum_{\vk \in \cK}  \sum_{\ell, r \in \cL_z}\(\frac{ F(\vk)}{\ell r}\).
$$
Combining~\eqref{eq:ML Bound 2} and~\eqref{eq:M SqSieve1}, we see that in any case we have
\begin{equation}
\label{eq:M SqSieve2} 
M    \le    \(K^m z^{-\alpha} + K^{m-1}     +  z^{-2} W\)z^{o(1)}.
\end{equation}

We further split the sum  $W$ into two sums as $W=U+V$, where
\begin{equation}
\label{eq:U and V}
\begin{split}
U &= \sum_{\substack{\ell, r\in \cL_z \\ P(\ell-1) = P (r-1)}}
 \sum_{\vk \in \cK}  \(\frac{ F(\vk)}{\ell r}\), \\
V &= \sum_{\substack{\ell,r\in\cL_z\\P(\ell-1)\ne P(r-1)}}  \sum_{\vk \in \cK}  \(\frac{ F(\vk)}{\ell r}\).
\end{split}
\end{equation}
To estimate $U$ (which also includes the diagonal case $\ell = r$), we use the trivial bound $(K+1)^m$ 
on each inner sum, deriving that
\begin{align*}
U & \le  (K+1)^m \sum_{\substack{\ell, r\in \cL_z\\ P(\ell-1)  = P(r-1)}} 1\\
& \le  (K+1)^m  \sum_{d\ge z^\alpha}
 \sum_{\substack{\ell, r\in \cL_z\\ \ell \equiv r\equiv 1 \bmod d}} 1
   \ll K^m \sum_{d\ge z^\alpha}\frac{z^2}{d^2}
 \ll K^m z^{2 -\alpha}. 
 \end{align*}
Hence we can write~\eqref{eq:M SqSieve2} as
\begin{equation}
\label{eq:M SqSieve3} 
M    \le    \(K^m z^{-\alpha} + K^{m-1}     +  z^{-2} V\)z^{o(1)}.
\end{equation}

\subsection{Bounds of character sums} 
To estimate $V$,  we first observe that  for every $\ell \in \cL_z$ the inequality
$\tau_\ell(g)\ge P(\ell -1)\ge \ell^\alpha$  implies (since $\alpha>\tfrac12$)
that $P(\ell-1)\mid \tau_\ell(g)$. 

Fix a pair  $(\ell, r) \in \cL_z{\times}\cL_z$ 
with $P(\ell-1) \ne P(r-1)$ and define
$$
h  =  \gcd\(\tau_{\ell}(g), \tau_r(g)\) \mand\vartheta=g^h.
$$
We observe that 
$$
\tau_\ell(\vartheta)= \tau_{\ell}(g)/h\mand
\tau_r(\vartheta)=\tau_{r}(g)/h.
$$
Furthermore, due to our choice of the set $\cL_z$ in 
Section~\ref{sec:Set Lz}  we have 
$$
\nu_2(\tau_{\ell}(g)) = \nu_2(\tau_{r}(g)) = \nu_2(h)
$$ 
and hence both $\tau_\ell(\vartheta)$ and $\tau_r(\vartheta)$ are odd.

We now write 
\begin{equation}
\label{eq:Tj} 
 \sum_{\vk \in \cK}  \(\frac{ F(\vk)}{\ell r}\)  = 
 \sum_{j_1, \ldots, j_m=1}^h T_{\ell, r} \(j_1, \ldots, j_m\),
\end{equation}
 where 
 \begin{align*}
T_{\ell, r} \(j_1, \ldots, j_m\)&=  \sum_{1 \le k_1 \le (K-j_1)/h}\\
 &\qquad  \ldots 
 \sum_{1 \le k_m \le (K-j_m)/h}
  \(\frac{ c_1 g^{k_1 h+j_1} +\ldots +c_m g^{k_m h+j_m}}{\ell r}\)  \\
&=  \sum_{1 \le k_1 \le (K-j_1)/h}\\
 &\qquad  \ldots 
 \sum_{1 \le k_m \le (K-j_m)/h}
  \(\frac{ c_1 g^{j_1} \vartheta^{k_1} +\ldots +c_mg^{j_m}  \vartheta^{k_m}}{\ell r}\)  . 
\end{align*}
We can certainly assume that $z$ is large enough so that 
$$\gcd(c_1\cdots c_m, \ell r )=1$$ 
for $\ell, r \in \cL_z$. Therefore 
Lemma~\ref{lem:Char Main} applies to $T_{\ell, r} \(j_1, \ldots, j_m\)$ and 
implies
 \begin{align*}
& T_{\ell, r} \(j_1, \ldots, j_m\) \\
& \qquad \ll (K/h)^m  \(\tau_\ell(\vartheta) \tau_r(\vartheta)\)^{-m+1} (\ell r)^{(m-1)/2} \\
 &\qquad \qquad  \quad   +  \((K/h)^{m-1}  \(\tau_\ell(\vartheta) \tau_r(\vartheta)\)^{-m+1}+1\) (\ell r)^{(m+1)/2}\( \log (\ell r) \)^m.
\end{align*}
Using 
$$
\tau_\ell(\vartheta) \tau_r(\vartheta) \gg z^{2 \alpha}  \mand  \ell r \ll z^{2}
$$
we see that 
 \begin{align*}
& T_{\ell, r} \(j_1, \ldots, j_m\) \\
& \quad \ \le \( (K/h)^m  z^{(m-1) (1 -2 \alpha)}  +\( (K/h)^{m-1}  z^{(m-1) (1 -2 \alpha)+2}+z^{m+1}\) \) z^{o(1)}  .
\end{align*}

Therefore, after the substitution in~\eqref{eq:Tj}  we obtain
 \begin{align*}
& \left| \sum_{\vk \in \cK}  \(\frac{ F(\vk)}{\ell r}\)  \right| \\
& \qquad \quad \le \(K^m  z^{(m-1) (1 -2 \alpha)}  + K^{m-1}  h z^{(m-1) (1 -2 \alpha)+2}+h^m z^{m+1} \) z^{o(1)}  .
\end{align*}
Since obviously $h \le\gcd(\ell-1, r-1)$, from the definition of $V$ in~\eqref{eq:U and V} and using ~\eqref{eq:Card L simple}, 
we now derive 
$$
V  \le  \(K^m  z^{(m-1) (1 -2 \alpha)}  + D_1 K^{m-1}   z^{(m-1) (1 -2 \alpha)}+ D_m z^{m-1} \) z^{2+o(1)},
$$
where $D_1$ and $D_m$ are as in Lemma~\ref{lem:Sum gcd}  and thus we get
$$
D_1 \le z^{2+o(1)}  \mand   D_m \le  z^{m+ \alpha - \alpha m + 1+o(1)}.
$$

Therefore 
$$
V  \le  \(K^m  z^{(m-1) (1 -2 \alpha)}  +   K^{m-1}   z^{(m-1) (1 -2 \alpha)+2}+   z^{2m+ \alpha - \alpha m } \) z^{2+o(1)}  .
$$
which after the substitution in~\eqref{eq:M SqSieve3} yields
\begin{align*} 
M    \le    (K^m z^{-\alpha} + K^{m-1}   &  +  K^m  z^{(m-1) (1 -2 \alpha)}  \\
& +   K^{m-1}   z^{(m-1) (1 -2 \alpha)+2}+  z^{2m+ \alpha - \alpha m})z^{o(1)}.
\end{align*}

\subsection{Optimisation} 
Clearly we can assume that
\begin{equation}
\label{eq:z K} 
z \le K^{1/(2-\alpha)}
\end{equation}
 as  otherwise the last term $z^{2m+ \alpha - \alpha m }$ 
exceeds the trivial bound $K^m$.  
Furthermore, for $m\ge 3$  and $\alpha$ as in~\eqref{eq:BH bound},  we have 
\begin{equation}
\label{eq:alpha more than half}
(m-1) (2 \alpha-1) > \alpha
\end{equation}
and hence (using that $z^{(m-1)(1-2\alpha)}<z^{-\alpha}$ due to the inequality~\eqref{eq:alpha more than half}), we can simplify the above bound as follows:
$$
M    \le    \(K^m z^{-\alpha}    + K^{m-1} +  K^{m-1}z^{2-(m-1)(2\alpha-1)}  +    z^{2m+ \alpha - \alpha m }  \)z^{o(1)}.
$$ 
Moreover, since we have~\eqref{eq:z K} and $\alpha<1$, we see that 
$$K^mz^{-\alpha}>K^{m-1}$$
which means that
\begin{equation}
\label{eq:M-bound gen} 
M\le \left(K^{m}z^{-\alpha}+K^{m-1}z^{2-(m-1)(2\alpha-1)}+z^{2m+\alpha-\alpha m}\right)z^{o(1)}.
\end{equation}

First we note that for $m \ge 5$, with  our choice of $\alpha=0.677$ in~\eqref{eq:BH bound} along with our assumption~\eqref{eq:z K}, 
we have that
$$K^mz^{-\alpha}>K^{m-1}z^{2-(m-1)(2\alpha-1)}$$
and thus the second term in~\eqref{eq:M-bound gen}  never dominates and we choose 
$$
z = K^{m/(2m+ 2\alpha - \alpha m)} 
$$ 
to balance the first and the third terms. Hence for   if $m\ge 5$, we obtain:
\begin{equation}
\label{eq:M-bound 5} 
M\le K^{m- m\alpha/(2m+2\alpha-\alpha m)+o(1)}.
\end{equation}

For $m=4$ and  with~\eqref{eq:BH bound}, direct calculations show that as for $m \ge 5$, it is better to balance the
first and the third terms in~\eqref{eq:M-bound gen}  (rather than  the first and the second terms), hence~\eqref{eq:M-bound 5} 
also holds for $m=4$. 

Finally,  for   $m=3$ one checks that balancing  the
first and the  second terms in~\eqref{eq:M-bound gen} with  $z=K^{1/(4-3\alpha)}$ 
leads to an optimal result 
$$M\le \(K^{3- \alpha/(4-3\alpha)} + K^{(6-2\alpha)/(4-3\alpha)}\)K^{o(1)}\le 
K^{3- \alpha/(4-3\alpha)+o(1)},$$
which  concludes the proof (see also \eqref{eq:K log}).

\section{Comments}
\label{sec:comm}

The proof of Theorem~\ref{thm:Qn g},  depends on the bound
 $$
\sum_{\vk \in \cK}  \omega_z\(F(\vk)\)\le  \sum_{\ell\in \cL_z}  T_m(K, \ell),
$$
where  $T_m(K, \ell)$ is the number of solutions  to the congruence 
$$
 F(\vk) \equiv 0 \pmod \ell, \qquad \vk   \in  \cK_m(K), 
 $$
 where $\cK_m(K)$ is given by~\eqref{eq:set KmK}, see the proof of  Lemma~\ref{lem:omega-z}. 
 In fact, in  the proof of  Lemma~\ref{lem:omega-z}  we use the trivial bound 
\begin{equation}
\label{eq:T-bound Triv} 
 T_m(K, \ell) \le (K+1)^{m-1} \(\frac{K+1}{\tau_\ell(g)} + 1\) \ll  K^m z^{-\alpha} + K^{m-1},
\end{equation}
 which holds for any $\ell \in \cL_z$ and is the best possible for $m=2$.
For $m \ge 3$ we get a better bound using exponential sum. This  does not improve 
our final result, however since it can be of independent interest and since it maybe becomes 
important if better bounds of $W$ in~\eqref{eq:M SqSieve1} become available (or maybe 
with some other modifications of the argument) we present such a better bound in
Appendix~\ref{app:A}, see Lemma~\ref{lem:Cong Fk}.

The method of the proof   of Theorem~\ref{thm:Qn g} also works  for relations of the form
$$ 
n^2 =  \sum_{i=1}^m c_i g_i^{k_i}, 
$$ 
with integer coefficients $c_1, \ldots, c_m$ of the same sign and  arbitrary 
integer  bases $g_1, \ldots, g_m \ge 2$. Indeed, 
in this case we still have a bound $O(\log N)$ on the exponents 
$k_1, \ldots, k_m$, which is important for our method. 
It is an interesting open question to establish such a bound for arbitrary 
$c_1, \ldots, c_m$.  Similarly, our method can also be used to estimate 
the number of $n \le N$ which can be represented as 
$$ 
n^2 =  u_1+\ldots+u_m
$$ 
for some {\it $\cS$-units\/} $u_1, \ldots , u_m$, that is, as a sum of $m$ integers 
which have all their prime factors from a prescribed finite set of primes $\cS$. Again, 
if negative values of $u_1, \ldots , u_m$ are allowed then some additional arguments 
are needed to bound the powers of primes in each $\cS$-unit. 

Furthermore, as in~\cite{BaSh} we observe that under the Generalised Riemann Hypothesis 
we can obtain a slightly large value of $\gamma_m$.  We now recall that $p$
is called a  {\it Sophie Germain prime\/} if $p$ and $2p+1$  are both prime. 
Under the assumption of the existence 
of  the expected number of Sophie Germain primes in intervals, or in fact of just $z^{1+o(1)}$ 
such primes up to $z$, we can choose a set $\cL_z$ in the argument of the proof 
of Theorem~\ref{thm:Qn g} with any $\alpha < 1$ and we see that under this assumption 
we can take $\gamma_m = m/(m+2)$ for $m \ge 3$.  

Finally, we note that other perfect powers $n^\nu$ for a fixed $\nu = 3,4, \ldots$, can be investigated 
by our method. However, one needs a version of a result of Baker and Harman~\cite{BaHa} 
for primes $\ell$ in the arithmetic progression $\ell \equiv 1 \pmod \nu$, so that there are
multiplicative characters modulo $\ell$ of order $\nu$. 

 \appendix

\section{Congruences with exponential functions}   
\label{app:A}

First we recall  the following special case of a classical result of Korobov~\cite[Lemma~2]{Kor2}.

\begin{lemma}
\label{lem:Exp Sum} 
Let $a  \in \Z$ and let $ \vartheta\in\Z$ with $ \vartheta\ge 2$.
Let   $\ell$    be a prime with 
$$
t = \tau_\ell(\vartheta), $$
and such that
$$
\gcd\(\ell , a \vartheta \) = 1.
$$ 
Then, we have
$$
\left| \sum_{k  =1}^{t} \e\(a \vartheta^k/t\) \right| \le \ell^{1/2} .
$$
\end{lemma}
 
 We now have a bound on $ T_m(K, \ell) $ which improves~\eqref{eq:T-bound Triv}  in some ranges.
 
\begin{lemma}
\label{lem:Cong Fk} Let $m \ge 3$.  Then for $K \ge z$ and $\ell \in \cL_z$, 
where  $\cL_z$ is as in Section~\ref{sec:Set Lz}, we have 
$$
 T_m(K, \ell) \ll  K^m z^{-1}+ K^m   z^{m/2 - \alpha(m-1)-1}. 
 $$
\end{lemma}

\begin{proof} Let $t=\tau_\ell(g)$; then since $\ell\in\cL_z$, we know that $t\ge z^\alpha$. We also let 
$$
 T_m(\ell) =  T_m(t-1 , \ell) .
$$
First we observe that $K \ge z \gg  t $  and thus 
\begin{equation}
\label{eq:TK Tt}
 T_m(K, \ell)  \le  \(\frac{K+1}{t} + 1\)^m T_m(\ell) \ll K^m t^{-m} T_m(\ell). 
\end{equation}
Now, using the orthogonality of exponential functions, we write
$$
 T_m(\ell) = \frac{1}{\ell}  \sum_{\vk \in \cK_m(t-1)} \sum_{a=0}^{\ell-1}  \e\(a F(\vk)/\ell\),
$$
where $\cK_m(t-1)$ consists of all $m$-tuples $(k_1,\dots,k_m)$ of non-negative  integers $k_i<t$. 
Now changing the order of summation, we obtain 
$$
 T_m(\ell) = \frac{1}{\ell} \sum_{a=0}^{\ell-1}   \sum_{\vk \in \cK_m(t-1)} \e\(a F(\vk)/\ell\)
 =    \frac{1}{\ell} \sum_{a=0}^{\ell-1}  \prod_{i=1}^m \sum_{k_i=0}^{t-1}   \e\(a c_i g^{k_i} /\ell\).
$$ 
The term corresponding to $a=0$ is equal to $t^m/\ell$. We can  
assume that $z$ is large enough (as otherwise the bound is trivial) so that $\gcd(c_1\cdots c_m, \ell) =1$
for $\ell \in \cL_z$. We apply now the bound of Lemma~\ref{lem:Exp Sum} to $m-2$ sums
over $k_3, \ldots, k_m$ and derive 
\begin{equation}
\label{eq:T Bound} 
 T_m(\ell)  \le t^m/\ell + \ell^{(m-2)/2}    \frac{1}{\ell} R, 
\end{equation}
where 
$$
R = 
  \sum_{a=0}^{\ell-1}  \left|\sum_{k_1=0}^{t-1}   \e\(a c_1 g^{k_1} /\ell\)\right| \cdot 
   \left| \sum_{k_2=0}^{t-1}   \e\(a c_2 g^{k_2} /\ell\)\right| 
$$
(note that after an application   of Lemma~\ref{lem:Exp Sum} we have added 
the term corresponding to $a=0$ back to the sum). 
By the Cauchy inequality 
$$
R^2  \le \left(\sum_{a=0}^{\ell-1}  \left|\sum_{k_1=0}^{t-1}   \e\(a c_1 g^{k_1} /\ell\)\right|^2\right)\cdot \left(
 \sum_{a=0}^{\ell-1}    \left| \sum_{k_2=0}^{t-1}   \e\(a c_2 g^{k_2} /\ell\)\right|^2\right). 
$$
Using the orthogonality of exponential functions again,  we derive 
$$
 \sum_{a=0}^{\ell-1}  \left|\sum_{k_1=0}^{t-1}   \e\(a c_1 g^{k_1} /\ell\)\right|^2 
 = \ell  t
$$
and similarly for the sum over $k_2$. Hence $R \le \ell  t$, and after substitution in~\eqref{eq:T Bound} 
we derive 
$$
 T_m(\ell)  \le  t^m/\ell + \ell^{(m-2)/2} t, 
$$
which together with the inequality~\eqref{eq:TK Tt} and the fact that $t\ge z^\alpha$ concludes the proof. 
\end{proof}

\section*{Acknowledgements}

D.~G. and S.~S. were partially supported by a Discovery Grant from NSERC, 
A.~O. by ARC Grants~DP180100201 and DP200100355, and I.~S.  by an ARC Grant~DP200100355.

\end{document}